\documentclass[12pt]{article}
\usepackage{amsmath,amsopn,amssymb,amsthm,amsfonts}
\usepackage[T2A]{fontenc}
\numberwithin{equation}{section}

\newtheorem{theorem}{Theorem}[section]
\newtheorem{corollary}[theorem]{Corollary}

\newtheorem{lemma}[theorem]{Lemma}

\def\Ker {{\rm Ker}}

\def\dim {{\rm dim}}
\def\div {{\rm div}}
\def\rot {{\rm rot}}

\def\grad {{\rm grad}}

\begin{document}

\vspace{7mm}

\centerline{\Large \bf Bi-invariant metric on contact diffeomorphisms group}

\vspace{3mm}

\centerline{\large \bf N.~K.~Smolentsev }

\vspace{7mm}

\begin{abstract}
We show the existence of a weak bi-invariant symmetric nondegenerate 2-form  on  the  contact diffeomorphisms  group $\mathcal{D}_\theta$ of a contact Riemannian manifold $(M,g,\theta)$ and study its properties. We describe the Euler's equation on a Lie algebra of group $\mathcal{D}_\theta$ and calculate the sectional curvature of $\mathcal{D}_\theta$. In a case $\dim M =3$ connection between the bi-invariant metric on $\mathcal{D}_\theta$ and the bi-invariant metric on volume-preserving diffeomorphisms  group $\mathcal{D}_\mu$ of $M^3$ is discover.
\end{abstract}

\section{Contact transformations group and bi-invariant metric}\label{Contact-D}
Let $M$ be a smooth (of class $C^\infty$) compact orientable manifold of dimension $n = 2m+1$ without boundary. The manifold $M$ is said to be \emph{contact} if a 1-form $\theta$ with the following property is given on it:  the $(2m+1)$-form $\theta \wedge(d\theta)^m$ does not vanish everywhere on $M$. Such  a  form $\theta$ is  said  to  be  \emph{contact}. A  vector  field $\xi$ is  said  to  be  \emph{characteristic}  if  it  has  the properties
$$
\theta(\xi) = 1,\qquad     d\theta(\xi,\cdot) = 0.
$$
A \emph{contact distribution} $E$ on $M$ is defined as the kernel of the form $\theta$, $E=\Ker \theta$. Clearly, $TM = E\oplus \mathbb{R}\xi$. Let $\Gamma(TM)$  be the space of smooth vector fields on $M$ and $\Gamma(E)$ be the space of smooth vector fields on $M$  belonging to a distribution $E$.

A Riemannian structure $g$ on $M$ is said to be \emph{associated} with $\theta$ if there exists a tensor field $\varphi$ of type (1,1) on $M$ such that for any vector fields $X$ and $Y$ on $M$,

(1) $g(X,\xi) =\theta(X)$;

(2)  $\varphi^2 = -\mathrm{I} + \theta \otimes\xi$;

(3)  $d\theta(X,Y) = g(X,\varphi Y)$, \\
where $\varphi$ is considered as a morphism $\varphi:TM \rightarrow TM$ and $\mathrm{I}$ is the identity morphism.

Let's note some additional properties \cite{Bla1}.  For any vector fields $X$ and $Y$ on $M$,

{\rm (4)} $g(\xi,\xi) = 1$;

{\rm (5)} the distribution $E$ is orthogonal to the field $\xi$,

{\rm (6)} $\varphi(\xi)=0$, $\varphi(E) = E$;

{\rm (7)} $\varphi$ is skew-symmetric and $\left(\varphi |_E \right)^2=-\mathrm{I}_E$;

{\rm (8)} $d\theta(\varphi X,\varphi Y)=d\theta(X,Y)$;

{\rm (9)} $g(X,Y)=\theta(X)\theta(Y)+d\theta(\varphi X,Y)$.

\vspace{1mm}
A mapping $\eta :M \rightarrow M$ is said to be \emph{contact} if it preserves the contact structure $\theta = 0$, i.e., if $\eta^* \theta = h\theta$, where $h$ is  a  function  of $M$. Although  the  algebra  of  infinitesimal  contact  transformations  is  usually defined as a subalgebra of the algebra of vector fields:
$$
\Gamma_\theta(TM) =\{V\in \Gamma(TM); \ L_V\theta=f\theta, \ \mbox { for a certain function }f  \},
$$
we define $\widetilde{\Gamma}_\theta$ as a subalgebra of $C^\infty(M) \oplus \Gamma(TM)$:
\begin{equation}
\widetilde{\Gamma}_\theta= \{(f,V)\in C^\infty(M) \oplus
\Gamma(TM);\ f\theta + di_V \theta +i_V d \theta =0 \}, 
\label{Eq-3.6}
\end{equation}
where the Lie algebra structure on $C^\infty(M) \oplus \Gamma(TM)$ is defined as follows:
\begin{equation}
[(f,U),(g,V)]= (U(g)-V(f), [U,V]). 
\label{[(f,U),(g,V)]-Eq-3.7}
\end{equation}
It is easy to verify that $\widetilde{\Gamma}_\theta$ is indeed a subalgebra of $C^\infty(M) \oplus \Gamma(TM)$.

\textbf{Remark 1.}  The algebra $C^\infty(M) \oplus \Gamma(TM)$ defined above is the Lie algebra of the semidirect product $C_*^\infty(M)*\mathcal{D}$ of the group $C_*^\infty(M)$ of positive smooth functions on $M$ and the diffeomorphisms group $\mathcal{D}$.
The group operation is given by
$$
(a,\eta)*(b,\zeta)= (a(b\circ \eta), \eta\circ \zeta).
$$
It is shown in \cite{Omo5} (Theorem 4.5.1) that $C_*^\infty(M)*\mathcal{D}$ is strongly an ILH-Lie group.  The coordinate mapping can be given by
\begin{equation}
\psi:C^\infty(M) \oplus \Gamma(TM)\rightarrow C_*^\infty(M) \oplus
\mathcal{D},\quad (f,V) \mapsto (e^f,E(V)),
\label{Eq-3.8}
\end{equation}
where $V \mapsto E(V)$ is the coordinate mapping on the diffeomorphisms group defined in \cite{Eb-Mar}.
\vspace{1mm}

There exists a one-to-one linear correspondence between $\widetilde{\Gamma}_\theta$ and the function space $C^\infty(M)$. Indeed, let $(f,U)\in \widetilde{\Gamma}_\theta$; then $L_U \theta+f\theta =0$. Decompose  the  field $U$ in  accordance  with  the  decomposition $TM = \mathbb{R}\xi\oplus E$:  $U=h\xi+v$.  Then
$$
L_U \theta+f\theta = dh +i_vd\theta + f\theta =0.
$$
Hence $\grad h -\varphi(v) +f \xi =0$.  We obtain from the latter relation that
$$
f= -\xi(h), \quad \varphi(\grad h) +v =0, \quad v =-\varphi(\grad h).
$$
Therefore,
\begin{equation}
(f,U)= (-\xi(h),h\xi-\varphi(\grad h)). 
\label{Eq-3.9}
\end{equation}
The latter formula defines the correspondence of $\widetilde{\Gamma}_\theta$  with the function space $C^\infty(M)$: $(f,U) \leftrightarrow h=\theta(U)$.

\begin{theorem} [see \cite{Omo5}, Theorem 8.3.6] \label{Th3.14}
The group
$$
\widetilde{\mathcal{D}}_\theta=\{(f,\eta)\in C_*^\infty(M)*\mathcal{D};\ f\eta^*\theta=\theta \}
$$
is a strong, closed ILH-subgroup of the group $C_*^\infty(M)*\mathcal{D}$ with the Lie algebra $\widetilde{\Gamma}_\theta$.
\end{theorem}

\vspace{1mm}
A contact form $\theta$ is said to be \emph{regular} \cite{Bla1} if the vector field $\xi$ induces a free action of the unit circle $S^1$ on the manifold $M$. In this case, the quotient manifold $N=M/S^1$ is defined; the 2-form $d\theta$ is lowered to $N$ and defines a symplectic structure $\omega$ on $N$.  If $\pi:M\rightarrow N$ is the natural projection, then $\pi^*\omega =d\theta$.

On $M$ we fix  the  \emph{associated  metric} $g$ and  the  \emph{affinor} $\varphi $ corresponding to it.   Assume that the contact structure $\theta$ on $M$ is regular.

\vspace{1mm}
The  transformation $\eta:M \rightarrow M$ is  said  to  be \emph{exactly  contact},  or  it  is  called  \emph{quantomorphism} if  it  preserves  the  contact  form $\theta$: $\eta^* \theta = \theta$.   Let $\mathcal{D}_{\theta}$ be  the  connected  component  of  the  group  of  all exact contact transformations of the manifold $M$:
$$
\mathcal{D}_{\theta}= \{\eta \in \mathcal{D}_0, \quad \eta^* \theta = \theta \}.
$$
For any $\eta\in \mathcal{D}_{\theta}$ we have:
$$
Ad_\eta \xi = \xi.
$$

Omori showed in \cite{Omo5} in the case of a regular contact
manifold $(M,\theta)$ that the group $\mathcal{D}_{\theta}$ is an ILH-Lie group.  In proving this, he used the fact that the group $\mathcal{D}_\theta$ consists of diffeomorphisms commuting with the free action of the compact group $S^1$ on $M$.

\begin{theorem} [see \cite{Omo5}]
The  group $\mathcal{D}_\theta$ is  a  strong  ILH-subgroup  of  the  group $\widetilde{\mathcal{D}}_\theta$. The  Lie  algebra  of  the  group $\mathcal{D}_\theta$ consists of contact vector fields $X$ on $M$, i.e., those fields for which $L_X \theta =0$.
\end{theorem}

Let $T_e \mathcal{D}_{\theta}=\{X\in \Gamma(TM);\ L_X \theta =0 \}$ be the Lie  algebra  of the  group $\mathcal{D}_{\theta}$.
If $X\in T_e \mathcal{D}_{\theta}$ is a contact vector field, then the function $f = \theta(X)$ is called the \emph{contact Hamiltonian} of the field $X$, and the field $X$ itself is usually denoted by $X_f$.  The condition $L_X \theta = 0$ immediately implies $i_X d\theta = -df$. Therefore, the function $f$ is constant on the trajectories of the characteristic vector field $\xi$.
Using the latter relation, it is easy to show that the contact vector field $X_f$ has the form
\begin{equation}
X_f = f\xi - \varphi\ \grad f. 
\label{X-f-11.1}
\end{equation}

Let us define the \emph{Lagrange bracket} $[f,g]$ of contact Hamiltonians $f$ and $g$ by
\begin{equation}
[f,g]=X_f(g). 
\label{[f,g]-11.2}
\end{equation}
Then the following relation holds for the Lie bracket of contact vector fields and the Lagrange bracket:
\begin{equation}
[X_f,X_g]=X_{[f,g]}.  
\label{[Xf,Xg]-11.3}
\end{equation}

The group $\mathcal{D}_{\theta}$ has the following natural right-invariant weak Riemannian structure:  if $X,Y \in T_e \mathcal{D}_{\theta}$, then
\begin{equation}
(X,Y)_e=\int_M g(X(x),Y(x)) d\mu(x), 
\label{(X,Y)-D-theta-11.4}
\end{equation}
where $\mu=\theta\wedge (d\theta)^n$ and at integration the form $\mu$ we denote as $d\mu$.

Introduce one more inner product on $T_e \mathcal{D}_{\theta}$:
\begin{equation}
\langle X_f ,X_h \rangle_e =\int_M f h\ d\mu. 
\label{<Xf,Yg>-D-theta-11.5}
\end{equation}

The  relation  between  the  natural  weak  Riemannian  structure (\ref{(X,Y)-D-theta-11.4})  and (\ref{<Xf,Yg>-D-theta-11.5}) is expressed by the formula
\begin{equation}
(X_f,X_h)_e=\langle X_{f+\triangle f},X_h \rangle_e, 
\label{(Xf,Xg)-D-theta-11.6}
\end{equation}
where $\Delta = - \div\circ\grad$ is the Laplacian. Indeed, we have
$$
(X_f,X_h)_e=\int_M g(X_f,Y_h) d\mu =\int_M g(f\,\xi - \varphi\, \grad f,h\,\xi - \varphi\, \grad h) d\mu =
$$
$$
=\int_M fh \, d\mu + \int_M g(\varphi\, \grad f,\varphi\, \grad h) d\mu =\int_M fh d\mu +\int_M  g(\grad f,\grad h) d\mu =
$$
$$
=\int_M fh d\mu +\int_M -\div (\grad f)\, h\, d\mu =\int_M (f+\Delta f) h\, d\mu.
$$

\begin{theorem}  \label{Biinv-D-theta}
The inner product (\ref{<Xf,Yg>-D-theta-11.5}) on the Lie algebra $T_e\mathcal{D}_{\theta}$ of the group $\mathcal{D}_{\theta}$ defines a bi-invariant weak Riemannian structure on the group $\mathcal{D}_{\theta}$.
\end{theorem}

\begin{proof}
It is easy to see that it is invariant under the adjoint action of the group $\mathcal{D}_{\theta}$ on $T_e \mathcal{D}_{\theta}$.  Indeed, if $X_f = f\xi - \varphi\, \grad f$, then
\begin{equation}
Ad_\eta X_f= (f\circ \eta^{-1})\xi- Ad_\eta(\varphi\, \grad f)= X_{f\circ \eta^{-1}}.
\label{Ad-eta-Xf}
\end{equation}
Moreover, $\eta^*(\mu)=\eta^*(\theta\wedge(d\theta)^m)=\theta\wedge(d\theta)^m =\mu$.
Therefore,
$$
\langle Ad_\eta (X_f), Ad_\eta (X_g) \rangle_e = \langle X_{f\circ \eta^{-1}}, X_{g\circ \eta^{-1}} \rangle_e=\int_M (f\circ\eta^{-1})( g\circ \eta^{-1})\ d\mu=
$$
$$
=\int_M f\cdot g\ \eta^*(d\mu) =\int_M f\cdot g\ d\mu = \langle X_f, X_g \rangle_e.
$$
Therefore, (\ref{<Xf,Yg>-D-theta-11.5}) defines a bi-invariant weak Riemannian structure on the group $\mathcal{D}_{\theta}$.
\end{proof}

\section{Euler equation on the Lie algebra $T_e \mathcal{D}_{\theta}$} \label{Euler}
Let $\mathfrak{g}$ be a semisimple, finite-dimensional Lie algebra,
and  let $H$ be  a  certain  function  on $\mathfrak{g}$. In \cite{Mi-Fo},  it  was  shown  that  the  extension  of  the  Euler  equation on  the  Lie  algebra  of  the  group $SO(n,\mathbb{R})$ of  motions  of  an $n$-dimensional  rigid  body  to  the  case  of  the general semisimple Lie algebra $\mathfrak{g}$ is an equation of the form
\begin{equation}
\frac{d}{dt}X = [X,\grad H(X)], 
\label{Eq-9.7}
\end{equation}
where $X \in\mathfrak{g}$ and the gradient of the Hamiltonian function $H$ is calculated with respect to the invariant Killing--Cartan inner product on $\mathfrak{g}$.

As $\mathfrak{g}$, let us  consider  the  Lie  algebra $T_e\mathcal{D}_{\theta}$ of contact vector fields  on  the regular contact Riemannian manifold $M$. On $T_e\mathcal{D}_{\theta}$, we have the invariant nondegenerate form (\ref{<Xf,Yg>-D-theta-11.5}) and the function (kinetic energy)
$$
T(X_f)= \frac 12(X_f,X_f)_e = \frac 12 \int_M g(X_f,X_f)d\mu,\quad X_f \in T_e\mathcal{D}_{\theta}.
$$
The function $T$ can be written as follows in terms of the inner product (\ref{<Xf,Yg>-D-theta-11.5}):
$$
T(X_f) = \frac 12(X_f,X_f)_e= \frac 12\langle X_{f+\triangle f},X_f \rangle_e.
$$
Perform the  Legendre  transform $Y_h = X_{f+\triangle f}$;  then
$$
T(X_f) =T(Y_{(1+\triangle)^{-1}h}) =\frac 12\left< Y_h, Y_{(1+\triangle)^{-1}h} \right>_e.
$$
Consider  the
Hamiltonian function $H(Y_h) =\frac 12 \left< Y_h, Y_{(1+\triangle)^{-1}h} \right>_e$ on the Lie algebra $T_e\mathcal{D}_{\theta}$. The gradient of the function $H$ with respect to the invariant inner product (\ref{<Xf,Yg>-D-theta-11.5}) is easily calculated:
$$
\grad H(Y_h) =  Y_{(1+\triangle)^{-1}h}.
$$
As in the finite-dimensional case, let us write the Euler equation on the Lie algebra $T_e\mathcal{D}_{\theta}$:
\begin{equation}
\frac{d}{dt}Y_h =[Y_h,Y_{(1+\triangle)^{-1}h}] = Y_{[h,(1+\triangle)^{-1}h]}. 
\label{Eq-9.8}
\end{equation}
\begin{equation}
\frac{d}{dt}h =[h,(1+\triangle)^{-1}h]. 
\label{Eq-Euler}
\end{equation}

Since $h = {f+\triangle f}$ and $[f,f]=0$, it follows that
\begin{equation}
\frac {\partial (f+\triangle f)}{\partial t}=[\triangle f,f] , 
\label{Eq-9.9}
\end{equation}

On the Lie algebra $T_e\mathcal{D}_{\theta}$, Eq.  (\ref{Eq-9.8}) or (\ref{Eq-9.9}) has the following two quadratic first integrals:
\begin{equation}
m(Y_h)=\left<Y_h,Y_h\right>_e=(X_{(1+\triangle)f},X_f)_e=\int_M g(X_{(1+\triangle)f},X_f) d\mu,
\label{m(Yg)}
\end{equation}
\begin{equation}
H(Y_g) = T(X_f) =\frac 12\int_M g(X_f,X_f) d\mu,
\label{H(Yg)}
\end{equation}
where $Y_h = X_{(1+\triangle)f}$.
The first of them $m(Y_h)$ is naturally called the \emph{kinetic moment}.  The invariance of the function $m$ follows from  the  invariance  of  the  inner  product  (\ref{<Xf,Yg>-D-theta-11.5})  on $T_e\mathcal{D}_{\theta}$. The  second  integral $H(Y_h)$ is  the \emph{kinetic energy}.

Since $H(X_f)=\frac 12\left<X_f,X_{(1+\triangle)f}\right>_e$, the operator $1+\triangle :T_e\mathcal{D}{\theta} \rightarrow T_e\mathcal{D}{\theta}$ is the inertia operator of our mechanical system $(T_e\mathcal{D}_{\theta}, H)$.
The eigenvectors $V_i$ of the operator $1+\triangle$ are naturally called (analogously to the rigid body  motion)  the \emph{axes  of  inertia},  and  the  eigenvalues $\lambda_i$ of  $1+\triangle$  are  called  the  \emph{moments  of  inertia}  with respect to the axes $V_i$.

The Euler equation (\ref{Eq-9.9}) can be written in the form
$$
\frac {\partial X_{(1+\triangle)f}}{\partial t}=-L_{X_f} X_{(1+\triangle)f},
$$
where $L_{X_f}$ is the Lie derivative. Therefore \cite{Ko-No}, the vector field $X_{(1+\triangle)f}(t)$ is transported by the flow $\eta_t$ of the field $X_f$:
$$
X_{(1+\triangle)f}(t)=d\eta_t(X_{(1+\triangle)f}(0)\circ \eta_t^{-1})=Ad_{\eta_t}(X_{(1+\triangle)f_0}) = X_{((1+\triangle)f_0)\circ\eta_t^{-1}},
$$
where $X_f(0)=X_{f_0}$ is the initial velocity field. Hamiltonian function $(1+\Delta) f(t,x)$ is  transported by the flow $\eta_t$:
$$
(1+\Delta) f(t,x)=((1+\Delta)f_0)(\eta_t^{-1}(x)).
$$
This immediately implies the following theorem.

\begin{theorem} \label{Th10.4}
Let $f= f(t,x)$ be a solution of the Euler equation (\ref{Eq-9.9}), and let $\eta_t$ be the flow on $M$ generated by the vector field $X_f$. Then the following quantities are independent of time $t$:
\begin{equation}
T=\frac 12\left(X_f,X_f \right)_e=\frac 12\int_M f(1+\Delta) f\ d\mu, 
\label{(L-Eq-10.9}
\end{equation}
\begin{equation}
I_k=\int_M \left((1+\Delta)f\right)^k\ d\mu. 
\label{(Ik-Eq-10.10}
\end{equation}
\end{theorem}

The curve $Y_h(t)= X_{(1+\triangle)f}(t)= Ad_{\eta_t}(X_{(1+\triangle)f_0})= Ad_{\eta_t}(Y_{h_0})$ on the Lie algebra $T_e\mathcal{D}_\theta$ that is a solution of the Euler equation (\ref{Eq-9.8}) lies on the orbit $\mathcal{O}(Y_{h_0})=\{Ad_\eta Y_{h_0};\ \eta \in \mathcal{D}_\theta \}$ of the
coadjoint action of the group $\mathcal{D}_\theta$.

Since the kinetic moment $m(X)$ is preserved under the motion, the orbit $\mathcal{O}(Y_{h_0})$ lies on the ''pseudo-sphere''
$$
S =\{Y_h \in T_e\mathcal{D}_\theta;\  \left<Y_h,Y_h\right>_e = r\},
$$
where $r=\left<Y_{h_0},Y_{h_0}\right>_e$. It  is  natural  to  expect  that  the  critical  points  of  the  function $H(Y_h)$ on  the orbit $\mathcal{O}(Y_{h_0})$ are stationary motions.  The orbit $\mathcal{O}(Y_h(0))$ is the image of the smooth mapping
$$
\mathcal{D}_\theta \rightarrow T_e\mathcal{D}_\theta,\quad  \eta \rightarrow Ad_\eta Y_{h_0}.
$$
Therefore, the tangent space $T_Y\mathcal{O}(Y_{h_0})$ to the orbit at a point $Y$ is
$$
T_Y\mathcal{O}(Y_{h_0})=\{[W,Y];\  W \in T_e \mathcal{D}_\theta\}.
$$
A point $Y\in \mathcal{O}(Y_{h_0})$ is critical for the function $H(Y_h)$ if $\grad H(Y_h)$ is orthogonal to the space $T_Y\mathcal{O}(Y_{h_0})$.
Since $\grad H(Y_h) =Y_{(1+\triangle)^{-1}h}$, the latter condition is equivalent to
$$
\left<Y_{(1+\triangle)^{-1}h},[W,Y]\right>_e=0\ \mbox{ for any }W\in T_e\mathcal{D}_\theta(M).
$$
The invariance of $\left< .,. \right>_e$ implies
$$
\left<[Y,Y_{(1+\triangle)^{-1}h}],W\right>_e =0\ \mbox{ for any }W\in T_e\mathcal{D}_\theta(M).
$$
From the nondegeneracy of the form $\left< .,. \right>_e$ on $T_e\mathcal{D}_\theta(M)$, we obtain that a point $Y$ on the orbit $\mathcal{O}(Y_h(0))$ is a critical point of the function $H$ iff
$$
[Y,Y_{(1+\triangle)^{-1}h}] = 0.
$$
It follows from the Euler equation $\frac {d}{dt}Y_h = [Y_h,Y_{(1+\triangle)^{-1}h}]$ that \emph{a field $Y_h$  is an equilibrium state of our system (i.e., a stationary motion) iff\, $Y_h$ is a critical point of the kinetic energy $H(Y_h)$  on the orbit $\mathcal{O}(Y_{h_0})$}.

\section{Curvature of the group $\mathcal{D}_{\theta}$} \label{Curvature}
The  sectional  curvature  of  the  group $\mathcal{D}_{\theta}$ with  respect  to  the  bi-invariant  metric (\ref{<Xf,Yg>-D-theta-11.5}) in  the  direction  to  a 2-plane $\sigma$ given  by  an  orthonormal  pair  of  contact  vector  fields $X_f,X_h \in T_e\mathcal{D}_{\theta}$ is expressed by the formula
$$
K_{\sigma}=\frac 14 \int_M [f,h]^2 d\mu.
$$
Thus, the group $\mathcal{D}_{\theta}$ is of nonnegative sectional curvature, and $K_\sigma=0$  iff the contact Hamiltonians $f$ and $h$ commutes, $[f,h] =0$.
\vspace{1mm}

The sectional curvature of the natural right-invariant Riemannian structure (\ref{(X,Y)-D-theta-11.4}) in  the  direction  to  a 2-plane $\sigma$ given  by  an  orthonormal  pair  of  contact  vector  fields $X,Y \in T_e\mathcal{D}_{\theta}$ can be found by the general formula \cite{Smo22}
\begin{multline}
\mathrm{K}_\sigma =-\frac 34([X,Y],[X,Y])_e -\frac 12([X,[X,Y]],Y)_e -\frac 12([Y,[Y,X]],X)_e - {} \\*
- \left(P_e(\nabla_XX),P_e(\nabla_YY)\right) + \frac 14\left(P_e(\nabla_XY + \nabla_YX),P_e(\nabla_XY +
\nabla_YX)\right). 
\label{K-sigma-Eq-8.4}
\end{multline}
where the operator $P_e$ must be considered as the orthogonal projection of the space $\Gamma(TM)$ on the space $T_e\mathcal{D}_{\theta}$ of contact vector fields on $M$ in accordance with the orthogonal decomposition $\Gamma(TM)= T_e\mathcal{D}_{\theta}\oplus \left(T_e\mathcal{D}_{\theta}\right)^{\bot}$.

\begin{lemma} \label{Lem11.1}
For any contact vector fields $X=X_f$ and $Y=X_h$ on $M$, the contact Hamiltonian $s$ of the vector field $P_e \left(\nabla_X Y \right)=X_s$ is expressed through $f$ and $h$ in the following way:
\begin{equation}
D s = \frac 12\left(D[f,h]+[f,D h]+[h, D f]\right), 
\label{D-s-D-theta-11.8}
\end{equation}
where $D = 1 + \Delta$, $\Delta$  is the Laplacian, and $[f,g]$ is the Lagrange bracket.
\end{lemma}

\begin{proof}
It was demonstrated in \cite{Eb-Mar} that, given the weak right-invariant Riemannian structure (\ref{(X,Y)-D-theta-11.4}) on $\mathcal{D}_{\theta}$, there exists a Riemannian connection whose covariant derivative $\widetilde{\nabla}$ at the identity $e\in \mathcal{D}_{\theta}$ is given by the formula
$$
(\widetilde{\nabla}_X Y)_e = P_e(\nabla_{X_e} Y_e),
$$
where $\nabla$ is the covariant derivative of the Riemannian connection on $M$ and $X(\eta)=X_e \circ \eta$, $Y(\eta)=Y_e \circ \eta$ are right-invariant vector fields on $\mathcal{D}_{\theta}$, $X_e, Y_e\in T_e\mathcal{D}_{\theta}$.
For determining $P_e(\nabla_{X_e} Y_e) = (\widetilde{\nabla}_X Y)_e$, we use the six-term formula
$$
2(\widetilde{\nabla}_X Y, Z)_e = X(Y,Z)+Y(Z,X)-Z(X,Y)+(Z,[X,Y])_e+(Y,[Z,X])_e-(X,[Y,Z])_e,
$$
where $X=X_f$, $Y=X_h$ and $Z=X_g$ are regarded as right-invariant vector fields on $\mathcal{D}_{\theta}$. Taking the right invariance of the weak Riemannian structure (\ref{(X,Y)-D-theta-11.4}) into account, we obtain $X(Y,Z)=Y(Z,X)=Z(X,Y)=0$. Using the bi-invariant scalar product (\ref{<Xf,Yg>-D-theta-11.5}), we obtain
$$
2(\widetilde{\nabla}_X Y, Z)_e =(X_g,[X_f,X_h])_e+(X_h,[X_g,X_f])_e-(X_f,[X_h,X_g])_e =
$$
$$
=(X_g,X_{[f,h]})_e+(X_h,X_{[g,f]})_e-(X_f,X_{[h,g]})_e =
\langle X_g,X_{D[f,h]}\rangle_e+\langle X_{Dh},X_{[g,f]}\rangle_e- \langle X_{Df},X_{[h,g]}\rangle_e =$$
$$
=\langle X_g,X_{D[f,h]}\rangle_e+\langle [X_f,X_{Dh}],X_{g}\rangle_e +\langle [X_h,X_{Df}],X_{g}\rangle_e =
$$
$$
=\langle X_{D[f,h]}+X_{[f,Dh]} + X_{[h,Df]},X_{g}\rangle_e.
$$
On the other hand, we have
$$
2(\widetilde{\nabla}_X Y, Z)_e =2(P_e(\nabla_{X}Y),Z)_e = 2(X_s,X_g)_e =2\langle X_{Ds},X_g\rangle_e.
$$
\end{proof}

\begin{corollary} \label{Cor11.2}
Let $X=X_f$ and $Y=X_h$. If $X_s= P_e \left(\nabla_X X \right)$ and $X_q =P_e \left(\nabla_X Y + \nabla_Y X\right)$, then
\begin{equation}
(1+\Delta) s =[f, \Delta f],\qquad (1+\Delta) q = [f,\Delta h]-[\Delta f,h]. 
\label{D-s-D-theta-11.9}
\end{equation}
\end{corollary}

From (\ref{K-sigma-Eq-8.4}) and Lemma \ref{Lem11.1} immediately implies the following theorem.

\begin{theorem} \label{Th11.3}
The sectional curvature of the group $\mathcal{D}_{\theta}$ with respect to metric  (\ref{(X,Y)-D-theta-11.4}) in the direction  of  a  2-plane $\sigma\subset T_e\mathcal{D}_{\theta}$ composed  of  an  orthonormal  pair  of  contact  vector fields $X_f, X_h \in T_e\mathcal{D}_{\theta}$ is expressed by the formula
\begin{multline}
K_{\sigma}= \frac 14 \int_M [f,h]^2 d\mu - \frac 34 \int_M [f,h]\Delta[f,h]d\mu + {} \\*
+ \frac 12\int_M [f,h]([\Delta f,h]+ [f,\Delta h])d\mu - \int_M
[f,\Delta f]\ D^{-1}([h,\Delta h])d\mu + {} \\*
+ \frac 14\int_M ([f,\Delta h]-[\Delta f,h])\ D^{-1}([f,\Delta h]-[\Delta f,h]) d\mu , 
\label{K-sigma-D-theta-11.10}
\end{multline}
where $D=1+\Delta$, $\Delta =-\div\circ \grad$ is the Laplacian, and $[f,h]$ is the Lagrange bracket of functions on a contact manifold.
\end{theorem}

In  the  case  where,   as  the  contact  Hamiltonians $f$ and $h$, we  choose  eigenfunctions  of  the  Laplace operator, $\Delta f=\alpha f$ and $\Delta h = \beta h$, formula (\ref{K-sigma-D-theta-11.10}) is considerably simpler:
\begin{equation}
K_{\sigma}=-\frac 34 \int_M [f,h]\Delta [f,g]\ d\mu + \frac {1+2(\alpha +\beta)}{4} \int_M [f,h]^2\ d\mu + \frac {(\alpha -\beta)^2}{4}\int_M [f,h]\ D^{-1}([f,h])d\mu. 
\label{K-sigma-D-theta-11.11}
\end{equation}

If the structural constants defined from
$$
[f,h] = c_{fh}^i f_i,
$$
where $f_i$ is the orthonormal system of eigenfunctions of the Laplace operator with eigenvalues $\alpha_i$, of the Lie algebra $T_e \mathcal{D}_\theta$ of contact vector fields are known, then the formula for the sectional curvatures becomes
\begin{equation}
K_{\sigma}=-\frac{1}{(1+\alpha)(1+\beta)}\left(-\frac 34
\sum_{i>0}\alpha_i (c_{fh}^i)^2+ \frac {1+2(\alpha +\beta)}{4} \sum_{i>0}(c_{fh}^i)^2+ \frac
{(\alpha -\beta)^2}{4}\sum_{i>0}\frac{(c_{fh}^i)^2}{1+\alpha_i}\right), 
\label{K-sigma-D-theta-11.12}
\end{equation}
here, it is assumed that the $L^2$-norms of the functions $f$ and $h$ are equal to 1.

Let us show that the calculation of the sectional curvatures of the group $\mathcal{D}_\theta$ by formula (\ref{K-sigma-D-theta-11.12}) reduces to the search for the structural constants of the Lie algebra of Hamiltonian vector fields on the symplectic manifold $N = M/S^1$. Recall  that  the  symplectic  structure  on $N$ is  defined  by  the  2-form $\omega$ uniquely found from the relation $\pi^* \omega = d\theta$, where $\pi$ is the projection of $M$ on $N = M/S^1$.

An exact contact transformation $\eta:M  \rightarrow M$ preserves the differential forms $\theta $ and $d\theta$. Therefore, the diffeomorphism $\eta$ preserves  the  characteristic  vector  field $\xi$ and  commutes  with  the  action  of $S^1$ on $M$. This implies that the contact diffeomorphism $\eta$ defines a diffeomorphism $\overline{\eta}$ of the manifold $N$ such that $\pi \circ \eta =\overline{\eta}\circ\pi$.  From the relation $\eta^* d\theta=d\theta$, we obtain $\overline{\eta}^*\omega = \omega$, i.e., $\overline{\eta}$ is a symplectic diffeomorphism of
the manifold $N$.  Therefore, we have obtained the homomorphism
$$
p: \mathcal{D}_{\theta} \rightarrow \mathcal{D}_\omega(N),\qquad
p(\eta) = \overline{\eta}.
$$
The  image $p(\mathcal{D}_{\theta,0})$ of  the  connected  component $\mathcal{D}_{\theta,0}$ of  the  identity  is  (see  \cite{Omo5}) a closed  ILH-Lie subgroup $\mathcal{G}$ of  exact  symplectic  transformations. Recall  that  the  Lie  algebra $\mathcal{H}=T_e\mathcal{G}$ of this  group consists of globally Hamiltonian vector fields on $N$.

   The sequence of groups
\begin{equation}
1 \rightarrow S^1 \rightarrow \mathcal{D}_{\theta,0} \rightarrow \mathcal{G} \rightarrow 1, 
\label{(exact-seq-D-theta-11.13}
\end{equation}
is exact. In the corresponding exact sequence of algebras
\begin{equation}
0 \rightarrow \mathbb{R} \rightarrow T_e \mathcal{D}_{\theta}
\rightarrow \mathcal{H} \rightarrow 0,      
\label{exact-seq-alg-theta-11.14}
\end{equation}
the  kernel  of  the  differential $dp$ is  the  one-dimensional  space  of  vector  fields  proportional  to  the  field $\xi$:\ $\Ker dp= \mathbb{R}\xi$.

In $T_e \mathcal{D}_{\theta}$, let us find the orthogonal complement to the kernel of the differential $dp$ with respect to the inner product (\ref{(X,Y)-D-theta-11.4}). The contact Hamiltonian of the characteristic vector field $\xi$ is identically equal to 1.
Let $X_f$ be any other contact vector field; then
$$
(\xi,X_f)=\int_M(1+\Delta(1))f\ d\mu = \int_M\ f\ d\mu =0.
$$
Thus, the contact vector field $X_f$ lies in the orthogonal complement to the vector field $\xi$  iff
$$
\int_M\ f\ d\mu =0.
$$

As we already noted, a contact vector field on $M$ has the form
$$
X_f = f\xi - \varphi\ \grad f,
$$
where $f$ is a smooth function on $M$ that is constant on the trajectories of the vector field $\xi$.
Therefore, it defines the function $F$ on $N$ by $f = F\circ \pi$.  Under the projection $\pi:M \rightarrow N$, the contact vector field $X_f$ is mapped into the Hamiltonian vector field $X_F$,\ $X_F = d\pi(X_f)$.  In this case, the Lie bracket $[X_f,X_h]$ passes to the Lie bracket $[X_F,X_H]$ of the corresponding Hamiltonian vector fields $X_F,\ X_H$ on $N$\ \cite{Ko-No}.
Therefore, the Poisson bracket $\{F,H\}$ on $N$ corresponds to the Lagrange bracket $[f,h]$: $[f,h]=\{F,H\}\circ \pi$.

Note  that if $X_f$ lies  in  the  orthogonal  complement  to $\mathbb{R}\xi$ in $T_e \mathcal{D}_{\theta}$, then  the Hamiltonian $F$ on $N$ corresponding to $f$ has the property $\int_M\ F\ d\mu =0$.

It follows from the above arguments that the structural constants $c_{fh}^i$ of the Lie algebra $T_e \mathcal{D}_{\theta}$,
$$
[f,h]=c_{fh}^i f_i,\qquad   X_f, X_h\in (\mathbb{R}\xi)^\bot,
$$
coincide with the structural constants $C_{FH}^i$ of the algebra $\mathcal{H}$,
$$
\{F,H\} =C_{FH}^i F_i.
$$

\textbf{Conclusion.}  \emph{The  structural  constants $C_{FH}^i$ of  the  algebra $H$ of Hamiltonian  vector  fields  on $N$ can be used  in  formula  (\ref{K-sigma-D-theta-11.12}) instead  of the  structural  constants $c_{fh}^i$ of  the  Lie  algebra  of  contact  vector  fields on $M$ for calculating the sectional curvatures of the group $\mathcal{D}_{\theta}$ in directions $\sigma$ orthogonal to the field $\xi$.}
\vspace{2mm}

\textbf{Remark 2.} Since  the  contact  vector  fields $X_f$ commute  with  the  characteristic  vector  field $\xi$,  i.e., $[\xi,X_f]=0$, it follows that the other structural constants of $T_e\mathcal{D}_{\theta}$ of the form $c_{\xi h}^i$ vanish. Therefore,  we have the following conclusion:

\emph{   the  sectional   curvature  of the group $\mathcal{D}_{\theta}$ in  directions $\sigma$ containing the characteristic  vector field $\xi$ vanishes.}

\vspace{2mm}
\textbf{Remark 3.}  Assume  that  the  associated  metric $g$ on $M$ is $K$-contact. This  means  that  the  metric tensor $g$ is  invariant  with  respect  to  the  action  of $S^1$ on $M$,\ $L_{\xi} g = 0$\ \cite{Bla1}.  Then on  $N=M/S^1$, we can  define  the  Riemannian  structure $\overline{g}$ with  respect  to  which  the  projection $\pi:M \rightarrow N=M/S^1$ is a Riemannian submersion.  If the contact Hamiltonian $f$ on $M$ is an eigenfunction of the Laplace operator, $\Delta f = \lambda f$, then the corresponding Hamiltonian $F$ on $N$ is also an eigenfunction of the Laplace operator on $N$, $\Delta F = \lambda F$. In the general case, this is not true. Therefore, the use of the structural constants $C_{FH}^i$ of the algebra $\mathcal{H}(N)$ for calculating the sectional curvatures of the group $\mathcal{D}_{\theta}$ can be especially effectively used in the case of the $K$-contact structure.
Therefore, we can calculate the sectional curvatures of the group $\mathcal{D}_{\theta}$ of contact transformations of the three-dimensional unit sphere $S^3$ in the basis corresponding to  the  Legendre  functions $Y_{lm}(z,\varphi)$ on  the  sphere $S^2$. The  corresponding  bundle $S^3 \rightarrow S^2$ is the Hopf bundle in this case.

\section{Case of a three-dimensional manifold $M^3$}\label{Bi-inv-Contact-Diff-M3}
Let $M$ be a three-dimensional contact manifold. In this case, the contact metric structure on $M$ also has the following properties in addition to properties (1)--(9) given in the beginning of this paper:

{\rm (10)} $\theta\wedge d\theta =\mu$ is the Riemannian volume element on $M$;

{\rm (11)} $d\theta(\varphi X,X)=1$ for any unit vector field $X$ belonging to the distribution $E$;

{\rm (12)} the triple of vectors $(\varphi X,X,\xi)$, $X \in E_x$ composes a positively oriented basis of the space $T_xM$;

{\rm (13)} $\varphi X = X\times \xi$, where $\times$ is the vector product on the three-dimensional Riemannian manifold $M$;

{\rm (14)} $*\theta = d\theta$,\  $*d\theta=\theta$, where $*$ is the Hodge operator.

The exact contact transformation $\eta$ of the manifold $M$ preserves the forms $\theta$ and $d\theta$ and hence preserves the volume element $\mu=\theta\wedge d\theta$. Therefore, the group $\mathcal{D}_\theta$ of exact contact diffeomorphisms is a subgroup of the group $\mathcal{D}_\mu(M^3)$ of diffeomorphisms preserving the volume element $\mu$.

Let $\mathcal{D}_{\mu\,\partial}(M^3)$ be the group of exact diffeomorphisms preserving the volume element $\mu$. The Lie algebra $T_e\mathcal{D}_{\mu\,\partial}(M^3)$ consists of  vector fields $X$ of a type $X=\rot Y$.
We will prove that if the contact structure $(M^3,\theta,\xi,g)$ is $K$-contact, then the group $\mathcal{D}_\theta$ is a subgroup in $\mathcal{D}_{\mu\,\partial}(M^3)$.

\begin{lemma} \label{Le11.4}
The characteristic vector field $\xi$ is an eigenvector of the operator $\rot$,
\begin{equation}
\rot\, \xi  = \xi.  
\label{rot-xi-D-theta-11.16}
\end{equation}
\end{lemma}

\begin{proof}
Indeed, let $\omega_X =g^{-1}X$ be  the  differential  1-form  corresponding  to  the  vector  field $X$ on $M$ by  the metric  tensor $g$. Then $\rot X$ is  found  from  the  relation $\omega_{\rot X}=*d\omega_X$, where $*$ is  the  Hodge  operator.
For the characteristic vector field, we have $\omega_\xi = \theta$.  Therefore,
$$
\omega_{\rot \xi}= *d\omega_\xi = * d\theta = \theta = \omega_\xi.
$$
\end{proof}

For the $K$-contact  structure,  the  metric  tensor $g$ and  the  affinor $\varphi$ are  invariant  with  respect  to  the action of  the  group $S^1$ generated  by  the  field $\xi$: $L_\xi g=0$, $L_\xi\varphi=0$,  on $M$.  As is known,  the  contact vector  fields  commute  with  the  field $\xi$, $[\xi ,X_f]=0$. Taking  into  account  that $X_f=f\xi -\varphi\, \grad\, f$ and $\xi(f)=0$,  we
immediately obtain from the latter relation that $[\xi ,\varphi\, \grad\, f]=0$.
In the case of the $K$-contact structure, it is easy to show that the gradients of the contact Hamiltonians $f$ also commute with $\xi$:
$$
[\,\xi, \grad\, f] = 0.
$$
Indeed, since $\grad\ f \in \Gamma(E)$, it follows:
$[\xi, \grad\, f] =[\xi,-\varphi^2 \grad\, f]= -L_\xi \varphi(\varphi\,
\grad\ f) -\varphi(L_\xi (\varphi\,\grad\, f))= -\varphi(L_\xi (\varphi\,\grad\, f)) = -\varphi[\,\xi,\varphi\,\grad\, f] = 0.$

In what follows, we assume that the contact metric structure $(M^3,\theta,\xi,g)$ is $K$-contact.

\begin{lemma} \label{Le11.5}
If $X_f = f\xi - \varphi\, \grad\ f$ is a contact vector field on $M^3$, then
\begin{equation}
\rot X_f = (f-\Delta f)\,\xi + \varphi\, \grad f, 
\label{rot-Xf-11.17}
\end{equation}
where $\Delta =-\div\circ\grad$ is the Laplacian. In particular,
$$
\rot f\xi = f\,\xi + \varphi\, \grad f,\qquad  \rot (\varphi\, \grad f) = \Delta f\,\xi.
$$
\end{lemma}

\begin{proof}
$$
\rot X_f = \rot(f\xi + \xi\times\grad f) =
\grad f\times \xi + f \rot\xi +[\grad f,\xi] + \xi\div(\grad\,
f)-\grad f\, \div \xi =
$$
$$
=\varphi\, \grad f + f\xi +(-\Delta f)\xi = (f -\Delta f)\xi + \varphi\, \grad f.
$$
Here, we have used the relations $\varphi X = X\times \xi$, $\div\, \xi = 0$, $[\xi,\grad f]=0$ and $\rot(X\times Y) = [Y,X] +X \div Y - Y \div X$.
\end{proof}

In  the  case  of  the  $K$-contact  metric,  the  operator $\rot^{-1}$ can  be  explicitly  calculated  on  the  contact vector fields.

\begin{lemma} \label{Le11.6}
If the contact vector field $X_f = f\xi - \varphi\, \grad f\in T_e \mathcal{D}_{\theta}(M^3)$ lies in the orthogonal complement to the vector field $\xi$ with respect to the inner product (\ref{(X,Y)-D-theta-11.4}), that is $\int_M f \,d\mu_g =0$, then
\begin{equation}
\rot^{-1} X_f =-f\xi +2\varphi\, \grad (\Delta^{-1} f), 
\label{rot(-1)Xf-11.18}
\end{equation}
where $h=\Delta^{-1} f$ there is a function, such that $\Delta h=f$, $\xi(h)=0$ and $\int_M h \,d\mu_g =0$.
\end{lemma}

\begin{proof}
A  direct  verification of the   relation $\rot(\rot^{-1} X_f)=X_f$:
$$
\rot (\rot^{-1} X_f) = \rot(-f\,\xi +2\,\grad (\Delta^{-1} f)\times \xi) =
$$
$$
=-\grad f\times \xi -f\, \rot\xi +2\,[\, \xi, \grad (\Delta^{-1} f)] +2\,\grad (\Delta^{-1} f)\, \div \xi -2\, \xi\,\div(\grad\,(\Delta^{-1} f)) =
$$
$$
=-\varphi\, \grad f -f\, \xi +2\,\xi\, (\Delta)(\Delta^{-1} f)
=-\varphi\, \grad f -f\, \xi +2\,\xi\, f=f\, \xi-\varphi\, \grad f =X_f.
$$
Here, we have used the relations: $\varphi X = X\times \xi$, $\div \xi = 0$, $[\xi,\grad f]=0$ and $\rot(X\times Y) = [Y,X] +X \div Y - Y \div X$.
Note   that   the  vector   field   $-f\xi +2\varphi \grad (\Delta^{-1} f)$ is divergence-free.
\end{proof}

\begin{corollary} \label{cor1}
If the contact vector field $X_f = f\xi - \varphi\, \grad f\in T_e \mathcal{D}_{\theta}(M^3)$ lies in the orthogonal complement to the vector field $\xi$, that is $\int_M f
 \,d\mu_g =0 $, then
\begin{equation}
X_f = \rot (-f\xi +2\varphi\, \grad (\Delta^{-1} f)). 
\label{Xf=rot}
\end{equation}
\end{corollary}

From the lemma \ref{Le11.4} and corollary \ref{cor1} follows:

\begin{corollary} \label{Subgroup}
If the contact structure $(M^3,\theta,\xi,g)$ is $K$-contact, then the group $\mathcal{D}_\theta$ is a subgroup in $\mathcal{D}_{\mu\,\partial}(M^3)$.
\end{corollary}

On the group $\mathcal{D}_{\mu\,\partial}(M^3)$ of exact diffeomorphisms preserving the volume element $\mu$ there exists the bi-invariant weak Riemannian structure \cite{Smo8}:
\begin{equation}
\left < X,Y \right >_e^{D_\mu} =\int_M g\left( \rot^{-1}X,Y \right)\ d\mu,\qquad X,Y\in T_e\mathcal{D}_\mu(M^3). 
\label{<X,Y>-D-mu}
\end{equation}


Therefore the group $\mathcal{D}_{\theta} \subset  \mathcal{D}_{\mu\,\partial}(M^3)$ inherits one more bi-invariant form
\begin{equation}
\left < X_f,X_h \right >_e^{D_\mu} =\int_M g\left( \rot^{-1}X_f,X_h \right)\ d\mu(x). 
\label{<X,Y>-D-theta-11.5}
\end{equation}


\begin{theorem} \label{Th11.7}
If the contact  metric  structure  on $M^3$ is $K$-contact,  then  the  bi-invariant  inner  products (\ref{<X,Y>-D-theta-11.5}) and  (\ref{<Xf,Yg>-D-theta-11.5}) on $\mathcal{D}_{\theta}(M^3)$  are connected by following relations:

If contact vector fields $X_f, X_h\in T_e \mathcal{D}_{\theta}(M^3)$ lies in the orthogonal complement to the vector field $\xi$ with respect to the inner product (\ref{(X,Y)-D-theta-11.4}), that is $\int_M f \,d\mu_g = \int_M h \,d\mu_g=0$, then
\begin{equation}
\left < X_f,X_h \right >_e^{D_\mu} = -3 \left < X_f,X_h \right >_e^{D_\theta}. 
\label{Eq-11.19}
\end{equation}
If  $X_f=\xi$ then for all $X_h\in T_e \mathcal{D}_{\theta}(M^3)$,
\begin{equation}
\left <\, \xi,X_h \right >_e^{D_\mu} =  \left <\, \xi,X_h \right >_e^{D_\theta}. 
\label{Eq-11.19xi}
\end{equation}
In particular,
$$
\left < \xi, \xi \right >_e^{D_\mu} =\left < \xi,\xi \right >_e^{D_\theta}=
\int_M  d\mu,
$$
and if $\int_M h \,d\mu_g=0$, then $\left < \xi, X_h \right >_e^{D_\mu} =\left < \xi,X_h \right >_e^{D_\theta}=0$.
\end{theorem}

\begin{proof}
Indeed, if the contact vector fields $X_f, X_h \in T_e \mathcal{D}_{\theta}(M^3)$ lies in the orthogonal complement to the vector field $\xi$ with respect to the inner product (\ref{(X,Y)-D-theta-11.4}), then
\begin{multline*}
\left < X_f,X_h \right >_e^{D_\mu} =\left(\rot^{-1}X_f,X_h\right)_e=
\left(-f\xi + 2\varphi\, \grad (\Delta^{-1} f), h\,\xi -\varphi\, \grad\, h \right)_e = {} \\*
=-(f\,\xi,h\,\xi)_e-2\left(\varphi\, \grad (\Delta^{-1} f),\varphi\, \grad\
h \right)_e =  -\int_M fh\ d\mu - 2\int_M g(\grad (\Delta^{-1} f),\grad\, h)\, d\mu = {} \\*
=-\int_M fh\ d\mu + 2\int_M \div (\grad (\Delta^{-1} f))h\ d\mu =
-\int_M fh\ d\mu - 2\int_M fh\ d\mu = -3 \left < X_f,X_h \right >_e^{D_\theta}.
\end{multline*}
  If $X_f=\xi$, then
$$
\left < \xi, X_h \right >_e^{D_\mu} =\left(\rot^{-1}\xi,X_h\right)_e =\left(\xi,h\,\xi -\varphi\, \grad\, h\right)_e=
\int_M h\ d\mu = \left < \xi,X_h \right >_e^{D_\theta}.
$$
$$
\left < \xi, \xi \right >_e^{D_\mu} =\left < \xi,\xi \right >_e^{D_\theta}=
\int_M  d\mu.
$$
\end{proof}


\end{document}